\documentclass{mfat}
\pagespan{81}{93}

\usepackage{mmap}
\usepackage[utf8]{inputenc}

\usepackage{oldlfont}
\usepackage{amssymb}
\usepackage{amsmath}
\usepackage{latexsym}

\newtheorem{theorem}{Theorem}
\newtheorem{prop}{Proposition}

\newtheorem{lemma}{Lemma}

\newtheorem{example}{Example}

\begin{document}

\title[Dynamical systems of conflict in terms of  structural measures]
    {Dynamical systems of conflict in terms of  structural measures}

\author{Volodymyr Koshmanenko}
\address{Institute of Mathematics, National Academy of
Sciences of Ukraine, 3 Teresh\-chenkivs'ka, Kyiv, 01601, Ukraine}
\email{koshman63@googlemail.com}

\author{Inga Verygina}
\address{National Technical University of
Ukraine "Kyiv Polytechnic Institute", 37 Prospect Pe\-re\-mo\-gy,
Kyiv, 03056, Ukraine}
\email{veringa@i.ua}

\date{28/04/2015;\ \  Revised 19/06/2015}
\dedicatory{This paper is dedicated to Yu. M. Berezansky on his
90th birthday}

\begin{abstract}
We investigate the dynamical systems  modeling  conflict processes
between a pair of opponents. We assume that opponents are given on
a common space by distributions (probability measures) having the
similar or self-similar structure. Our main result states the
existence
  of the controlled conflict in which one of the opponents occupies almost whole conflicting
  space. Besides, we compare   conflicting effects
stipulated  by the rough structural approximation under
 controlled redistributions of starting measures.
\end{abstract}

\maketitle

\section{Introduction}
We study the emergence of space redistributions arising due to the
conflict interaction between opponents (opposite sides, players). The
roles of opponents may be played by various natural entities with
alternative trends (for examples see \cite{BKS,CVNW,Ep,KarKosh08}). In
turn, the conflicting space may appear as a territory, a living
resource, an ordering queue, in general, any value admitting division
(see \cite{AKS,Ch,H,KS(Eng11),Mu}).

We begin with an observation that opponents (competing sides) of
conflict processes usually are presented in a form of a similar or
self-similar structure: cells, bacterias, trees, peoples, etc.
That is why the description of opponents in many-dimensional terms is
a more adequate in comparison with single-meaning values of their
mutual powers. By this reason we propose in constructions of the
conflict theory to use the non-deterministic ideology of Quantum
Mechanics.  In particular, we will describe the states of opponents
in terms of distributions of probability measures.  Furthermore, we
will assume that these distributions have a certain similar or
self-similar geometrical structure consistent with a preassigned
division of the conflicting space.

Let us explain our approach in more details (see also
\cite{ABodK,BKK,BKS,KarKosh08}--\cite{KD},\cite{KoIF09}).

Denote by $A, B$ a couple of opponents (alternative sides) living on a
common resource space $\Omega$. In what follows $\Omega$ is a compact
of some metric space with the Borel $\sigma$-algebra $\mathcal{B}$ of
its subsets. Let $\lambda$ be a fixed $\sigma$-additive measure on
$\mathcal{B}$ such that $\lambda(\Omega)=1$.  In the simplest case one
can think that $\Omega=[0,1]$ and $\lambda$ denotes the usual Lebesgue
measure.

We denote by ${\mathcal M}(\Omega)$ a family of all $\sigma$-additive
finite signed measures on $\mathcal{B}$.  The subset of positive
measures is denoted by ${\mathcal M}^+(\Omega)$.  For a probability
measure $\mu $ we write $\mu \in {\mathcal M}_1^+(\Omega)$.

We suppose that opponents $A, B$ at the initial moment of time are
presented on $\Omega$ by a couple of different probability measures
$\mu, \nu \in {\mathcal M}_1^+(\Omega)$. The conflict interaction
between $A, B$ is represented by a discrete (or continuous) binary
mapping $\divideontimes$ in ${\mathcal M}_1^+(\Omega)$:
$$\mu\divideontimes\nu=\mu', \quad\nu\divideontimes\mu=\nu',$$
which usually is non-commutative and nonlinear.  We call each triple
$\{\Omega, {\mathcal M}_1^+(\Omega), \divideontimes \}$ the {\it
  dynamical system of conflict} (DSC).

The main problem of  the conflict theory is  to study and  describe
the behavior of trajectories
  of  DSC in  terms of couples of probability measures:
\begin{equation}\label{1tr}
  \left\{ \begin{array} {ll} \mu \\
           \nu \end{array}
   \right\}\stackrel{\divideontimes, t}{\longrightarrow }
\left\{ \begin{array} {ll} \mu(t) \\
           \nu(t) \end{array}
   \right\}, \quad  \mu, \nu\in  {\cal M}^+_1, \quad t\geq 0.
\end{equation}
We suppose that the time evolution of DSC in the general case is
governed by the system of differential equations
\begin{equation}\label{1deq} \frac{d \mu}{dt} =\mu \Theta- \tau, \quad \frac{d \nu}{dt}  =\nu \Theta- \tau, \end{equation}
where the conflict exponent $\Theta=\Theta(\mu,\nu)$ is a positive
quadratic form which describes the power of conflict interaction at
whole, and the measure-valued function $\tau=\tau(\mu,\nu)$
corresponds to the local confrontation between $A,B$. In what follows
we come to the discrete time $t= N=0,1,\ldots $ and use the system of
difference equations
\begin{equation}\label{1discrt}
 \left\{ \begin{array} {ll}  \mu^{N+1}(E)=\mu^{N}(E)+\mu^{N}(E)\Theta^{N} - \tau^{N}(E), \\
 \nu^{N+1}(E) = \nu^{N}(E) +\nu^{N}(E) \Theta^{N} - \tau^{N}(E),  \quad   E\in {\cal B}, \end{array}
 \right.
\end{equation}
where  we  omit   normalization denominators.

In \cite{Ko14} it was proved (see also \cite{KoTC,KoTC1,KoPe15}) that
each trajectory $\{\mu^N, \nu^N \}, N\geq 1$ starting with any couple
of probability measures $\mu, \nu \in {\mathcal M}_1^+(\Omega), \
\mu\neq \nu$, converges in the weak sense to a limit fixed point
$\{\mu^\infty, \nu^\infty\}$. This point creates an equilibrium state
for the system and is a compromise in the sense that $\mu^\infty \perp
\nu^\infty$. Moreover, for each dynamical system of conflict
$\{\Omega, {\mathcal M}_1^+(\Omega), \divideontimes \}$ given by
(\ref{1discrt}), there exists the limit $\omega$-set $\Gamma^\infty$
\cite{SMR}. It is an attractor consisting of all couples of mutually
singular measures from ${\mathcal M}_1^+(\Omega)$. Thus,
$$\Gamma^\infty= \{  \{\mu^\infty, \nu^\infty\} \ | \ \mu^\infty, \nu^\infty \in
{\mathcal M}_1^+(\Omega), \ \mu^\infty \perp \nu^\infty  \}.$$  It was proved also
 that each limit
state is  uniquely determined  by the starting couple  $\{\mu,
\nu\}$, and moreover,
$$\mu^\infty=\mu_+, \quad \nu^\infty=\nu_-,$$
where $\mu_+, \nu_-$ denote the    normalized components of a classic Hanh--Jordan
decomposition \cite{BUS,DSH,ShGu}
 of the signed measure $\omega=\mu-\nu=\omega_+-\omega_-$. Thus
\begin{equation}\label{mnpl}
\mu^\infty = \frac{\omega_+}{\omega_+(\Omega)}=: \mu_+, \quad
 \nu^\infty = \frac{\omega_-}{\omega_-(\Omega)}=: \nu_-.\end{equation}

In the simplest situation  the  dynamical system  of conflict can
 be written in terms of coordinates of  stochastic vectors
$p,r\in {\mathbb R}^n_+, n \geq 2$  corresponding to  opponent
sides,
$$ p_i^{N+1}=1/z^N(p_i^N \Theta^N- \tau_i^N), \quad  r_i^{N+1}=1/z^N(r_i^N \Theta^N- \tau_i^N), \quad  i=1,\ldots ,n.$$
Here we set $\Theta^N=(p^N,r^N)$ to be the inner product between the
vectors $p^N,r^N$ and $\tau_i^N=p_i^Nr_i^N$.  It was proved in
\cite{KoTC,KoTC1} that each trajectory $\{p^N, r^N\}_{N=0}^\infty$
starting with a couple of stochastic vectors $\{p^0=p, r^0=r\}, \
p\neq r$ converges with $N\to \infty$ to a fixed point $\{p^\infty,
r^\infty\}$ which creates a compromise state, $p^\infty \perp
r^\infty$.  This state is uniquely determined by the starting couple
$\{p, r\}$ and has an explicit coordinate representation
\begin{equation}\label{lfstv}
\begin{gathered}
p^\infty_i=\frac{d_i}{D}>0, \quad  i\in N_+, \quad r^\infty_k=-\frac{d_k}{D}>0, \quad  k\in N_-,
\quad
D=1/2\sum_{i=1}^n |d_i|,
\\
 p^\infty_i=0, \quad  i\notin N_+,\quad r^\infty_k=0, \quad  k\notin N_-,
\end{gathered}
\end{equation}
where $d_i=p_i-r_i$ and $N_+=\{i : d_i>0 \}$, \ $N_-=\{k :
d_k<0\}$. In \cite{KarKosh08,KD,Ko07,KS(Eng11)} we ge\-ne\-ralized the
above constructions to cases of piece-wise uniformly distributed
measures, self-similar, and similar structure measures.

In the present paper we study more specific questions connected with
the property of similar structure measures $\mu,\nu$ to have the weak
approximation in terms of piece-wise uniformly distributed measures
$$\mu=\lim_{k\to \infty}\mu_k, \quad \nu=\lim_{k\to \infty}\nu_k.$$
We analyze the effects which are produced by the rough structural
approximation and controlled redistribution (see below) of the
starting measures.  In other words we are interesting in all possible
spatial and valued changes of the limit measures when $\mu_k,\nu_k$
are subject to the structural redistributions, $\mu_k \to \tilde
\mu_k, \nu_k \to \tilde \nu_k$, so that the limits $${\tilde
  \mu}^\infty=\lim_{k\to \infty}{\tilde \mu}_k^\infty,\quad {\tilde
  \nu}^\infty= \lim_{k\to \infty}{\tilde \nu}_k^\infty$$ become
essentially different from $\mu^\infty, \nu^\infty$.

Our main result (it is hypothetical in the general case)  reads as follows.

 Given   a couple of similar structure measures
 $\mu,\nu \in {\cal M}^{\rm ss}(\Omega)$ with $
{\rm supp} \mu=\Omega={\rm supp} \nu$, for any $0<\varepsilon <1 $
 there exists a  controlled structural redistribution of the measure $\mu$ such
 that on $ 1\leq k < \infty$ step of the rough approximation,  $\mu_k \to \tilde \mu_k, $
 The limit conflict state $\{ \tilde \mu_{k}^\infty, \nu_{k}^\infty \}$ obeys the
 properties
 $$\lambda({\rm supp} {\tilde \mu}_{k}^\infty) \geq 1-\varepsilon,
 \quad
 {\tilde \mu}_{k}^\infty \perp \nu_{k}^\infty.$$

 It means that the controlled conflict under an appropriate strategy
 may lead to an expansion over most part of the territory.

 In the paper we prove only a simplest version of this observation
 (see Theorem \ref{6thg}).

 The paper consists of five sections. In Section 2 we briefly recall a
 general picture of DSC in terms of probability measures, in Section 3
 the notions of similar and self-similar structure measures are
 presented, Section 4 contains the main results, and finally, in the
 last section we discuss the interpretation of the obtained results
 and their possible applications.

\section{On dynamical systems of conflict  in terms of  abstract measures}

Let us recall a general scheme of our approach to the conflict theory
in terms of abstract measures (for more details, see
\cite{Ko14,KoPe15}). We will deal with dynamical systems $\{\Omega,
{\mathcal M}_{1, {\rm ac}}^+(\Omega), \divideontimes \}$ of natural
conflict (here ${\mathcal M}_{1, {\rm ac}}^+$ denotes a class of
absolutely continuous measures). The term "natural"  means that a
conflict composition $\divideontimes$ is defined by some fixed law of
the conflict interaction between opponents and their strategies do not
change during the time evolution.  In Sec. 4 we will discuss the
dynamical systems with rough approximations and controlled
redistributions when $\divideontimes$ is subjected to the strategical
changes.

Let us consider an abstract variant of DSC with the discrete time
$$\mu^{N+1}=\mu^N\divideontimes\nu^N , \quad \nu^{N+1}=\nu^N \divideontimes \mu^N, \quad N=0,1,\ldots $$
Their  state trajectories
\begin{equation}\label{DSK}
  \left\{ \begin{array} {ll} \mu^N \\
           \nu^N \end{array}
   \right\}\stackrel{\divideontimes}{\longrightarrow }
\left\{ \begin{array} {ll} \mu^{N+1} \\
           \nu^{N+1} \end{array}
   \right\}, \quad N=0,1,\ldots
\end{equation}
are governed by
the following {\it  law of conflict dynamic}:
 \begin{equation}\label{2drdsc}
  \left\{ \begin{array} {ll}  \mu^{N+1}(E) = \frac{1}{z^N}[\mu^N(E)
\big( \Theta^{N}+1 \big) - \tau^{N}(E)], \\
\nu^{N+1}(E) = \frac{1}{z^N}[\nu^{N}(E)
\big(\Theta^{N}+1\big) - \tau^{N}(E)], \quad  E\in {\mathcal B}, \end{array}
\right.
\end{equation}
where the measures $ \mu^0=\mu, \ \nu^0=\nu$ correspond to the initial
state.  The conflict exponent $\Theta^N$ in (\ref{2drdsc}) is defined
as
$$\label{thN}\Theta^N=\int_\Omega\int_\Omega {\cal K}(x,y)\varphi^N(x)\psi^N(y)\,dxdy,$$
where ${\cal K}(x,y)$ denotes the kernel of some positive bounded
operator $K$ in $L^2(\Omega, \ d\lambda)$
and $$\varphi^N(x)=\sqrt{\rho^N(x)}, \quad
\psi^N(x)=\sqrt{\sigma^N(x)},$$ where $\rho^N(x), \sigma^N(x)$ are the
Radon-Nikodym derivatives of $\mu^N,\nu^N$ with respect to $\lambda$.
Thus,
$$\label{thN}\Theta_N=(K\varphi^N,\psi^N)_{L^2(\Omega, \ d\lambda)},$$
Further, $\tau^N$ in (\ref{2drdsc}) stands for the occupation measure.
Its values characterize the presence of opponents on opposite
territories. By definition,
\begin{equation}\label{tauN} \tau^N(E) = \nu^{N}(E_+)+\mu^{N}(E_-), \quad E_+=E\bigcap\Omega_+,
  \quad E_-=E\bigcap\Omega_-,\end{equation}
where $\Omega=\Omega_-\bigcup \Omega_+$ corresponds to the Hahn--Jordan decomposition
(see \cite{DSH,ShGu}) of the
starting signed measure $\omega=\mu-\nu$. Finally, the normalizing denominator in
(\ref{2drdsc}) is defined as
 $$z^N=\Theta^N+1 -W^N, \quad  W^N=\mu^N(\Omega_-) + \nu^N(\Omega_+).$$
 It is easy to see that all measures  $\mu^N, \nu^N, \  N\geq 1$ in  (\ref{2drdsc})  are
 absolutely continuous and  probability, i.e., $\mu^{N}, \nu^{N} \in {\mathcal M}^+_{1, {\rm
 ac}}(\Omega)$.

 The DSC  defined by
(\ref{2drdsc}) has two separate  sets  of  fixed points. The first set contains all couples  of
 identical  measures from ${\mathcal M}^+_{1, {\rm ac}}(\Omega)$. Indeed, if $\mu=\nu$, then
  $\Theta^N$ is a constant for all $N$ and $\mu^N(E)=\mu(E)=\nu^N(E)=\nu(E)$ for each $ E\in {\mathcal B}$.
   The second set is composed of measures $\mu,\nu\in {\mathcal M}^+_{1, {\rm ac}}(\Omega)$
   which are orthogonal,
$\mu\perp \nu.$ In this  case
$\tau^N=0=W^N$ and $\Theta^N+1=z^N$   for all $N$. Due to (\ref{2drdsc}) we  find that
  $\mu^N=\mu, \ \nu^N=\nu$.

  In all other cases, when the starting measures are different,
  $\mu\neq\nu$, and mutually non-singular the following theorem is
  true.

\begin{theorem}\label{3thCnlt} 
  Let $\{\Omega, {\mathcal M}_{1, {\rm ac}}^+(\Omega), \divideontimes
  \}$ be a DSC generated by the system of difference equations
  (\ref{2drdsc}). Then each its trajectory (\ref{DSK}) starting with a
  couple of probability measures $\mu^0=\mu, \nu^0=\nu \in {\mathcal
    M}^+_{1, {\rm ac}}(\Omega),$ \ \ $\mu\neq\nu$ \ converges to a
  fixed point corresponding to a limit state $\{ \mu^\infty,
  \nu^\infty \}$ with
   \begin{equation}
\mu^\infty(E) = \lim_{N\to \infty}\mu^N(E), \quad \nu^\infty (E) = \lim_{N\to \infty}\nu^N (E), \quad E\in {\mathcal B},
\end{equation}
 where
\begin{equation}\label{mr}
\mu^\infty(E) =\frac{\mu(E_+)-\nu(E_+)}{D}= \mu_+(E), \quad \nu^\infty(E) =- \frac{\mu(E_-)-
\nu(E_-)}{D}= \nu_-(E)
\end{equation}
with $\mu_+, \nu_-$ defined by (\ref{mnpl}).
\end{theorem}

In (\ref{mr})   $D=1/2\int_\Omega |\rho(x)-\sigma(x)|dx$ stands for the total difference
 between measures $\mu,\nu$.

 If $\Omega$ is a finite set and $\mu, \nu$ are stochastic vectors
 $p,r\in {\mathbb R}^n_{+,1}$, then coordinates of the limit states
 $\{p^\infty, r^\infty\}$ are described by (\ref{lfstv}).  We refer to
 \cite{KoTC,KoTC1,Ko14}) for the proof of this theorem.

\section{The similar and self-simillar structure measures}
Let   $\Omega$ be a compact set of some metric space and let  $\lambda$  be a fixed $\sigma$-additive
measure on the Borel algebra of subsets. We suppose that $\lambda(\Omega)=1$.
Consider a specific class of measures as follows.

Fix $n>1$, assume $\Omega$ is consecutively divided onto non-empty
subsets (regions) on each $k$th level
\begin{equation}\label{1dq}
\Omega=\bigcup\limits_{{i_1}=1}^{n}\Omega_{{i_1}}
=\bigcup\limits_{{i_1},{i_2}=1}^{n}\Omega_{{i_1}{i_2}}
=\bigcup\limits_{{i_1},{i_2}, \ldots ,{i_k}=1}^{n}\Omega_{{i_1}{i_2}\ldots {i_k}}=\cdots,
\quad k=1,2,\ldots
\end{equation} such that all ratios
\begin{equation}\label{1raq}
q_{k{i_k}}:=\frac{\lambda(\Omega_{{i_1}{i_2}\ldots {i_k}})}{\lambda(\Omega_{{i_1}{i_2}\ldots {i_{k-1}}})}
, \quad  k\geq 1
\end{equation}
are independent of indices ${i_1},{i_2},{\ldots },{i_{k-1}},$ where $
\Omega_{i_0}=\Omega$.  Besides we suppose
that $$\inf\limits_{k,i_k}\{q_{ki_k}\}>0 \quad {\rm and} \quad
\sum\limits_{{i_k}=1}^{n}{q_{k{i_k}}}=1.$$ In what follows we fix some
division of $\Omega$ with above properties.

 We say that a probability measure $\mu$ from ${\cal M}^+_1(\Omega)$ belongs to the
 class of {\it similar structure measures}, \  write $\mu \in {\cal M}^{\rm ss}(\Omega)$,
  if, besides (\ref{1raq}),   all ratios
\begin{equation}\label{1rap} p_{k{i_k}}:=\frac{\mu(\Omega_{i_1\ldots i_k})}{\mu(\Omega_{i_1\ldots i_{k-1}})},
  \quad  k \geq 1 \end{equation}
with $\mu(\Omega_{i_1\ldots i_{k-1})}\neq 0$ are  independent of the indices
${i_1},{i_2},{\ldots },{i_{k-1}}.$  We recall that $\mu(\Omega_{i_0})=\mu(\Omega)=1$.
By this definition, $$ p_{k{i_k}}\geq0, \quad
\sum\limits_{{i_k}=1}^{n}{p_{k{i_k}}}=1, \quad k\geq 1. $$
In the particular case, where $q_{ki_k}$  and $p_{ki_k}$ in (\ref{1raq}) and (\ref{1rap})
do not depend on $k$, we say that the measure $\mu$ has the {\it self-similar
  structure},
\ write $\mu\in {\cal M}^{\rm sss}(\Omega)$.
Clearly ${\cal M}^{\rm sss}(\Omega)\subset {\cal M}^{\rm ss}(\Omega)$.

We say that a measure $\mu$ has the {\it partly similar structure},
write $\mu \in {\cal M}^{\rm pss}(\Omega)$, if conditions (\ref{1rap})
hold only up to some finite $k< \infty$. Inside the subsets
$\Omega_{i_1\ldots i_{k}}$, these measures have arbitrary
distributions.

\begin{lemma}\label{2asm} Each similar structure measure $\mu \in
  {\cal M}^{\rm ss}(\Omega)$
 is uniquely associated with a matrix
$$P=\{p_k\}_{k=1}^\infty=\{ p_{k i_k}\}_{i_k=1, k=1}^{n,
  \infty},\quad p_k\in {\mathbb R}_{+,1}^{n}, $$ such that all vectors
$p_k=(p_{k 1}, p_{k 2},\ldots ,p_{k n}), \ k\geq 1$ are stochastic,
i.e., $$ p_{k i_k} \geq 0,\quad p_{k 1}+p_{k 2}+\cdots +p_{k n}=1.$$
In particular, if there exists $k_0$ such that all $p_k, \ k\geq k_0$,
coincide with some fixed vector $p\in {\mathbb R}_{+,1}^{n}$, then
$\mu \in {\cal M}^{\rm sss}(\Omega)$.

 Each  partly similar structural measure  $\mu \in {\cal M}^{\rm pss}(\Omega)$ is uniquely
 associated with the finite
 matrix $$P_{k_0}=\{p_k\}_{k=1}^{k_0}=\{ p_{k i_k}\}_{i_k=1, k=1}^{n,
 k_0}, \quad  p_k\in {\mathbb R}_{+,1}^{n},\quad  k_0 < \infty.$$
 \end{lemma}

\begin{proof}
  Given $P$ define the measure $\mu_k$ for each $1\leq k<\infty$ as
  follows: $$\mu_k(\Omega_{{i_1}\ldots {i_l}}):=p_{1{i_1}}\ldots
  p_{l{i_l}}\equiv p_{i_1\ldots i_l}, \quad 1\leq l \leq k,$$ and put
  $\mu_k$ uniformly distributed inside subsets $\Omega_{{i_1} \ldots
    {i_k}}$. By this construction,
$$\frac{\mu_k(\Omega_{i_1\ldots i_{l}})}{\mu_k(\Omega_{i_1\ldots
i_{l-1}})}=\frac{ p_{i_1\ldots i_{l}}}{ p_{i_1 \ldots
i_{l-1}}}=p_{l{i_l}}, \quad  1 \leq l \leq k$$ and therefore
$\mu_k \in {\cal M}^{\rm pss}(\Omega)$. Thus, $\mu_k$ is a
piece-wise uniformly distributed measure and one can write for any
Borel set $E\in {\cal B}$: $$\mu_k(E)=\sum\limits_{{i_1}, \ldots
,{i_k}=1}^{n} \rho_{{i_1}\ldots {i_k}} \lambda_{{i_1}\ldots
{i_k}}(E),
 \ $$
where $$\rho_{{i_1}\ldots {i_k}}:= \frac{p_{i_1\ldots i_k}}{q_{i_1
\ldots i_k}},\quad
 q_{i_1\ldots i_k}:=\lambda(\Omega_{{i_1}\ldots {i_k}})=q_{1{i_1}}\ldots q_{k{i_k}},
 \quad  \lambda_{{i_1}\ldots {i_k}}=\lambda \upharpoonright  \Omega_{{i_1}\ldots {i_k}}.$$
We put $\mu(\Omega_{{i_1}\ldots {i_l}}):=\mu_k(\Omega_{{i_1}
\ldots {i_l}}), \  1 \leq l \leq k$ and for an arbitrary $E$
define $\mu$ by the weak limit  $$\mu(E)=\lim_{k\to
\infty}\mu_k(E), \quad  E\in {\cal B}.$$
 By this construction,
 $$\mu(\Omega_{{i_1}\ldots {i_{k-1}}})=\mu_k(\Omega_{{i_1}\ldots {i_{k-1}}})
 =\sum_{i_k=1}^n \mu_k(\Omega_{{i_1}\ldots {i_{k}}})=p_{i_1\ldots i_{i-1}}\sum_{i_k=1}^n p_{ki_k}, \quad  k\geq 1$$
 and therefore $\mu\in {\cal M}^{\rm ss}(\Omega)$.

Vice versa, if $\mu \in {\cal M}^{\rm ss}(\Omega)$, then
elements  $p_{k,i_k}$ (see (\ref{1rap})) define  the matrix $P$
with above properties.
\end{proof}

In what follows we will consider a couple of measures $\mu, \nu \in
{\cal M}^{\rm ss}(\Omega)$ associated with the matrices
\begin{equation}\label{PR} P=\{p_k \}_{k=1}^\infty =\{ p_{k i_k}\}_{k,i_k=1}^{ \infty, n}
\quad {\rm and} \quad R=\{r_k \}_{k=1}^\infty =\{ r_{k
i_k}\}_{k,i_k=1}^{ \infty, n}.  \end{equation}

In the case where $\Omega=[0,1]$, the above measures $\mu_k$ have
densities $\rho_k(x)$ which are the simple functions:
$$\rho_k(x)=p_{{i_1}\ldots {i_k}}\chi_ {\Omega_{i_1 \ldots
    i_k}}(x),\quad x \in [0,1],$$ where $\chi_ {\Omega_{i_1\ldots
    i_k}}(x)$ denotes the characteristic function of subsets
$\Omega_{i_1\ldots i_k}$.  The corresponding distribution functions
$F_k(x)$ of $\mu_k$ are continuous piece-wise linear ones increasing
from zero to 1.  Obviously, the sequence $F_k(x), k=1,2,\ldots $ is
point-wise convergent, i.e., there exists a left continuous function
$$F(x)=\lim_{k\rightarrow\infty} {F_k(x)}$$ which defines the
distribution function of the measure $\mu$ on $[0,1]$.
 Moreover, for  continuous $\mu$  this convergence is uniform,
  since all $F_k(x)$ are uniformly bounded.

\section{\bf The  rough  controlled conflicts}

Here we consider several variants of the rough and controlled conflict
interaction in terms of the structural measures. Theorem \ref{3thCnlt}
gives a general result of the mathematical theory of natural
(non-controlled) conflict. However, conflict confrontations in the
real situation occur usually with various deformations of the starting
data.  So, rather often the conflict actions are based on rather rough
estimates of mutual forces, their approximate distributions, and
relations.  That is why there appear rough redistributions of
opponents positions.  We are aimed to study different effects arising
from these reasons and describe a more adequate picture produced by
the approximation method in terms of the rough controlled conflict.
The language of the similar structure measures provides the excellent
tool for this aim.  Simultaneously, some kind of notion of the
controlled conflict appears naturally in this way.

We begin with a short review of the abstract mathematical scheme of
the conflict theory.  Let at the initial moment of time opponents
$A,B$ are distributed along the resource conflicting space $\Omega$ in
according with probability measures $\mu, \nu \in {\cal
  M}^+_1(\Omega)$.  Assume the law of conflict interaction is fixed by
equations (\ref{2drdsc}). Thus, there appears the dynamical system of
conflict $\left\{ \Omega, {\cal M}^+_1(\Omega), \divideontimes
\right\}$.  Note that in general it is not easy to find explicitly
the limit distributions $\mu^\infty, \nu^\infty$ described by Theorem
\ref{3thCnlt}. The reason is that the Hanh--Jordan decomposition
$\Omega=\Omega_+\bigcup \Omega_-$ is non-constructive. It appears as a
result of some approximate procedure (see \cite{DSH}).

Nevertheless, if the conflict space $\Omega$ is divided into a finite
set of regions (see (\ref{1dq})), one can fulfill the rough conflict
"program" using firstly only one step of division:
$\Omega=\bigcup_{i_1=1}^n\Omega_{i_1}$ and changing the starting
measures $\mu,\nu$ into piece-wise uniformly distributed ones
$\mu_1,\nu_1$.  Then one can come to the second more deep step of the
rough "program", and so on.  Importantly, the just described way of
the rough approximation of the conflict interaction is often realized
in applications.

Let us describe this approach in more detail. Let $\mu,\nu \in {\cal
  M}^{\rm ss}(\Omega), \ \ \mu\neq\nu$. Let us consider at the first
step of the rough approximation the piece-wise uniformly distributed
measures $\mu_1,\nu_1$ defined as follows:
   \begin{equation}\label{3str1} \mu_1(\Omega_{i_1})=\mu(\Omega_{i_1})= p_{1i_1},
    \quad \nu_1(\Omega_{i_1})=\nu(\Omega_{i_1})= r_{1i_1}, \quad i_i=1,2,\ldots ,n,
\end{equation}
where $p_{1i_1}, \ r_{1i_1}$ are coordinates of the vectors $p_1, r_1$
(see (\ref{PR})). The Hahn--Jordan decomposition
$\Omega=\Omega_{+,1}\bigcup\Omega_{-,1}$ corresponding to the signed
measure $\omega_1=\mu_1-\nu_1$ has the form
$$\Omega_{+,1}=\bigcup_{i_1\in N_{+,1}} \Omega_{i_1}, \quad
\Omega_{-,1}=\bigcup_{i_1\in N_{-,1}} \Omega_{i_1},$$ where
$$N_{+,1}=\{i_1 \ | \ \mu(\Omega_{i_1}) > \nu(\Omega_{i_1}), \quad
N_{-,1}=\{i_1 \ | \
 \mu(\Omega_{i_1}) < \nu(\Omega_{i_1}) \} \}.$$ We assume for simplicity  that  the regions
$\Omega_{i_1}$ with  $\mu(\Omega_{i_1}) = \nu(\Omega_{i_1})$ are absent.
Thus, due to Theorem \ref{3thCnlt} we have
\begin{equation}\label{4ift}
 \mu^\infty_1(\Omega_{i_1})=  \left\{ \begin{array} {ll}  \frac{p_{1i_1}-r_{1i_1}}{D_1}, \ i_1\in N_{+,1} \\
0, \ i_1\notin N_{+,1}  \end{array} \right. ,  \quad
 \nu^\infty_1(\Omega_{i_1})=  \left\{ \begin{array} {ll}  0, \ i_1\notin N_{-,1} \\
-\frac{p_{1i_1}-r_{1i_1}}{D_1}, \ i_1\in N_{-,1}  \end{array}
\right. ,
\end{equation}
where $D_1=1/2\sum_{i_1} |p_{1i_1}-r_{1i_1}|$.

Let $\mu_1(\Omega_{s})>0$ for some $i_1=s$ such that $s \in
N_{-,1}$. Then $\mu^\infty_1(\Omega_{s})=0$. Thus, the region
$\Omega_{s}$ is played over for the opponent $A$ if the conflict game
occurred at the level of the first rough approximation. This zero
distribution for $\mu^\infty_1$ on $\Omega_{s}$ appears due to the
starting inequality
 $$\mu_1(\Omega_{s})< \nu_1(\Omega_{s}),
 \quad s\in N_{-,1}.$$ Nevertheless, possibly there exists a subset
 $\tilde \Omega_{s}\subset \Omega_{s}$ with the opposite
 inequality $$\mu_1(\tilde \Omega_{s}> \nu_1(\tilde \Omega_{s}).$$

 We are interested in the folowing question. What is the maximal
 Lebesgue measure of such a subset $\tilde \Omega_{s}$ ?  This subset
 can be saved for the opponent $A$ under the next steps of
 approximation with a more thin division of the conflict space or
 under using the controlled redistribution (for the definition see
 below) of the $\mu_1$ inside $ \Omega_{s}$. More precisely, we are
 interested in the following question. What is the biggest (in the
 sense of Lebesgue measure) subset $\tilde \Omega_{s} \subset
 \Omega_{s}$ such that at the second step of approximation, when $
 \Omega_{s}$ is subjected to the division $\Omega_{s}=\tilde
 \Omega_{s}\bigcup \tilde \Omega_{s}^{\rm c}$, for the controlled
 redistribution $\mu_2 \to \tilde \mu_2$, the inequality
$$ \tilde \mu_2(\tilde \Omega_{s})>\nu(\tilde \Omega_{s}) $$
holds?

We define the procedure of the controlled redistribution as follows.
We say that a measure $\mu\in {\cal M}_1^+(\Omega)$ is subjected to
a controlled redistribution along a set $\Omega_s \subset\Omega$, if
it is replaced by the measure $\tilde \mu\in {\cal M}_1^+(\Omega)$
which differs from $\mu$ only inside $\Omega_s$.

One of the effects under the rough approximation with the consequent
controlled redistributions we formulate as follows.
\begin{theorem}\label{4thbig}
  Let $\mu, \nu \in {\cal M}^{\rm ss}(\Omega)$ and $\mu_1, \nu_1$ be
  defined by (\ref{3str1}) at the first step of the rough structural
  approximation. Assume that
$$0< \mu_1(\Omega_{s})< \nu_1(\Omega_{s}), \quad s=i_1 \in N_{-,1}
\ $$
 and therefore $\mu^\infty_1(\Omega_{s})=0, \ \nu^\infty_1(\Omega_{s})>0$.
Let for  the  division
\begin{equation}\label{4div} \Omega_{s}=\tilde \Omega_{s}\bigcup \tilde \Omega_{s}^{\rm c},
\end{equation}
$\tilde \mu_2, \nu_2=\nu_1$ denote the measures at the second step of
the rough structural approximation with the controlled redistribution
of $\mu_2$ obeying the inequalities
\begin{equation}\label{4cdiv} \mu_1(\Omega_{s}) \geq \tilde \mu_2(\tilde \Omega_{s}) >
  \nu_2(\tilde \Omega_{s}).\end{equation}
Then $$ \tilde \mu^\infty_2(\tilde \Omega_{s}) >0, \quad
\nu^\infty_2(\tilde \Omega_{s}) =0$$ and moreover
\begin{equation}\label{4ast}
  \lambda(\tilde \Omega_{s})\leq \sigma_s(\mu,\nu) \lambda(\Omega_{s}) \quad {with}
  \quad
  \sigma_s(\mu,\nu)=
  \frac{\mu_1(\Omega_{s})}{\nu_1(\Omega_{s})}.\end{equation} In the
extremal case, where the value  $\lambda(\tilde \Omega_{s})$ is
maximal,  both limiting distributions   on $\tilde \Omega_{s}$ for
opponents $A, B$ are zero,   $$\tilde \mu_2^\infty(\tilde
\Omega_{s})= \nu_2^\infty(\tilde \Omega_{s})=0.$$
 \end{theorem}

\begin{proof}
  To prove inequality (\ref{4ast}) we will apply the geometrical
  reasoning using the uniform distribution of $\nu_2$ on
  $\Omega_s$. It is easy to see that there exists a non-unique
  division (\ref{4div}) such that $\nu_2(\tilde \Omega_{s})\equiv
  \nu_1(\tilde \Omega_{s}) < \tilde \mu_2(\tilde \Omega_{s})=\mu_1(
  \Omega_{s})$, where we produce a controlled redistribution $\mu_1
  \to \tilde \mu_2$ such that $\tilde \mu_2(\tilde \Omega_{s}^c)=0$
  and $\tilde \mu_2(\tilde \Omega_{s})= \mu_1( \Omega_{s})$. Now the
  estimate (\ref{4ast}) appears by the linear geometrical
  interpolation.  Indeed, take any subset $\tilde \Omega_s\subset
  \Omega_s$ such that $\nu_1(\tilde \Omega_{s})\leq \mu_1(\tilde
  \Omega_{s})$ and put ${\tilde \mu}_2(\tilde
  \Omega_{s})=\mu_1(\Omega_s), \ \ {\tilde \mu}_2(\tilde
  \Omega_{s}^c)=0$. Denote
  $\sigma_{s}=\sigma_{s}(\nu,\lambda):=\nu_1(\Omega_s)/\lambda(\Omega_s)$.
  Then obviously
   $$\nu_2(\tilde \Omega_{s})=\nu_1(\tilde \Omega_{s})=\sigma_{s}\lambda(\tilde \Omega_s)\leq
   \tilde \mu(\tilde \Omega_s)=\mu_1(\Omega_s).$$
   Therefore
  $$\lambda(\tilde \Omega_{s})\leq \mu(\Omega_s)/\sigma_{s}
  =\frac{\mu_1(\Omega_{s})}{\nu_1(\Omega_{s})} \lambda(\Omega_{s}).$$
 This proves (\ref{4ast}).

 Clearly, in (\ref{4ast}) we have the equality iff $\nu_2(\tilde
 \Omega_{s})= \mu(\Omega_{s})= \tilde \mu_2(\tilde \Omega_{s})$.  In
 this extremal case
 \begin{equation}\label{4sup} \sup \lambda(\tilde \Omega_{s})
   =\frac{\mu_1(\Omega_{s})}{\nu_1(\Omega_{s})} \lambda(\Omega_{s}),\end{equation}
 where  the supremum  is taken over all divisions (\ref{4div}) satisfying
 conditions (\ref{4cdiv}). Then both sequences
 $\tilde \mu_2^N(\tilde \Omega_{s}), \  \nu_1^N(\tilde \Omega_{s}), \ N=1,2,...$
 converge  to zero due to  (\ref{4ift}) (see the vector version of Theorem \ref{3thCnlt}).
\end{proof}

Therefore, if the opponent $A$ associated with a measure $\mu$ at the
first step of the structural approximation looses some region
$\Omega_s$, i.e., $\mu^\infty_1(\Omega_{s})=0$, then at the second
step of structural approximation (under an additional division of
$\Omega_s$) using the controlled redistribution inside $\Omega_s$ it
can return a part $\tilde \Omega_{s}$ of this region. The size of the
returned subset (in the sense Lebesgue measure is estimated by
(\ref{4ast})). In turn, the opponent $B$ wins the whole region
$\Omega_s$ at the first step, however at the second step it gets zero
inside $\tilde \Omega_s$ since it does not produce any preserving
actions: $\nu_2^\infty(\tilde \Omega_{s}) =0, \ \ \tilde
\mu^\infty_2(\tilde \Omega_{s}) >0.$

It is not easy to generalize Theorem \ref{4thbig} to the case of
arbitrary similar structure measures since the local densities
$\sigma_{i_1\ldots i_k}: =\nu(\Omega_{i_1\ldots
  i_k})/\lambda(\Omega_{i_1\ldots i_k})$ take in general any
values. That is why it is so hard to predict in applications the
evolution of space redistributions when both opponents use different
individual strategies under the controlled conflict interactions.

Let us consider a couple of examples.

\begin{example}\label{4Ex2} From losses to gains, or an expansion on a
  new territory.  \end{example} Let
$\Omega=\bigcup_{i=1}^3\Omega_{1_i}, \ \lambda(\Omega_{1_i})=1/3$. Let
us put in correspondence to opponents $A,B$ the measures $\mu_1,
\nu_1$ such that vectors $p_1=\{\mu_1(\Omega_{i_1})\}, r_1=
\{\nu_1(\Omega_{i_1})\}$ from ${\mathbb R}^3_{+,1}$ have the following
coordinates: $$p_{i_1}=1/3, \quad i_1=1,2,3, \quad
r_1=(2-\varepsilon)/9,\quad r_2=(1+\varepsilon)/3,\quad r_3=4/9, \quad
\varepsilon>0.$$ Then, by using Theorem \ref{3thCnlt}, we find by
direct calculation that
 $$\Omega_-=\Omega_3, \quad
 \nu_1^\infty(\Omega_-)=1, \quad \lambda(\Omega_-)=1/3.$$ In
 particular, the opponent $A$ has a priority for two regions
 $\Omega_+=\Omega_1 \bigcup \Omega_2$ with $\lambda(\Omega_+)=2/3$.
 However, at the second step of partition,
 $\Omega=\bigcup_{i_1,i_2=1}^3 \Omega_{i_1i_2}$, the measures $\mu_2,
 \nu_2$ have new, different signs of their priorities and therefore the
 Hanh--Jordan decomposition is changed $$\Omega_-=
 \Omega_{22}\bigcup\Omega_{23}\bigcup\Omega_{32}\bigcup\Omega_{33},
 \quad \lambda(\Omega_-)=4/9. $$ Thus, $\nu^\infty_2(\Omega_-)=1$ and
 the area of priority for the opponent $B$ becomes larger. We can go
 to the next step of approximation and put $$
 r_1=(9-3\varepsilon)/27,\quad r_2=(9+\varepsilon)/27, \quad
 r_3=(9+2\varepsilon)/27.  $$ It leads to a greater extension of the
 area of priority for $B$.

 This example shows that the strategy of the directed priority:
 $r_1<r_2<r_3$ in comparison with the strategy of uniform distribution
 $p_i=1/3$ leads to an extension of the occupation area under the
 conflict interaction on the way of increasing steps of the structural
 approximation.

 \begin{example}\label{4Ex} Spectral gaps  as a result of  conflict
   interactions. \end{example} Let $\mu, \nu\in {\cal M}^{\rm
   sss}(\Omega)$. Assume that at the first step of the rough
 approximation the measures $\mu_1, \nu_1$ are presented by vectors
 $p_1,r_1\in {\mathbb R}^n_{+,1}, \ n\geq 3$ with the coordinates
$$p_{i_1}=1/n, \quad i_1=1,\ldots ,n, \quad r_s=(n-1)/n, \quad
r_{i_1\neq s}=1/(n(n-1)), \quad  1\leq s\leq n.$$ By this
assumption, all $p_{i_1}>r_{i_1}, i_1\neq s$, \  and $p_s<r_s$.
Thus,
 $D_1=r_s-p_s= (n-2)/n$ and
$$p_{i_1}^\infty=1/(n-1),  \quad r_{i_1}^\infty=0, \quad {\rm if}
\quad i_1\neq s, \quad
  {\rm and}
\quad p_s=0, \quad r_s= 1.$$ Therefore, the opponent $A$ loses the
region $\Omega_s$ and  $B$ wins this region  with probability 1.
The spectral support of $\mu^\infty$ coincides with $
\Omega_s^c=\Omega\backslash \Omega_s$.

Let us come to the second step of the structural approximation.  Then
we get
$$p_{i_1i_2}=1/n^2 \quad {\rm for \  all} \quad i_1,i_2=1,\ldots
,n$$ and $$ r_{i_1i_2}=1/(n-1)n, \quad  {\rm if \  both} \quad
i_1, i_2 \neq s,$$ $$r_{i_1i_2}=1/n^2, \quad {\rm if \  only \
one} \quad i_1 \ {\rm or} \   i_2 \neq s, $$ $$ r_{ss}=
((n-1)/n)^2.$$ Now $D_2=D_1=1-2/n$, \
$p_{i_1i_2}^\infty=1/(n-1)^2$,
 \ $r^\infty_{ss}=1,$ and $ r^\infty_{i_1i_2}=0, \ {\rm for \ \ the \
 rest} \ i_1, i_2.$ In particular,
  $$p^\infty_{i_1s}=p^\infty_{si_2}=r^\infty_{si_2}=r^\infty_{i_1s}=0, \quad i_1\neq s\neq i_2.$$
  Thus, at the second step of the rough approximation we observe gaps
  in the regions $\Omega_{i_1s}, \Omega_{si_2} \subset \Omega_{i_1}, \
  i_1\neq s$ for the opponent $A$ and gaps for the opponent $B$ in the
  regions $\Omega_{si_2} \subset \Omega_{s}, \ i_2\neq s$. If we will
  continue the rough approximation to the third step,  similar gaps
  appear for the opponent $B$ in regions $\Omega_{ssi_3} \subset
  \Omega_{ss}, \ i_3\neq s$.  And so on.

The general result   in this direction reads as follows.

\begin{theorem}\label{4sss} Let $\mu, \nu\in {\cal M}^{\rm sss}(\Omega)$.
  Assume $\mu\neq\nu$ and let $\Omega_s$ be such that
$$0< \mu_1(\Omega_{s})< \nu_1(\Omega_{s}), \ \ s\in N_{-,1},$$
 and therefore  $\mu^\infty_1(\Omega_{s})=0, \ \nu^\infty_1(\Omega_{s})>0$. Then
 with necessity  there exist  subsets $\Omega_{si_2\ldots i_k}\subset\Omega_{s}$ such that the
 opposite inequality holds
$$ \mu_k(\Omega_{si_2\ldots i_k})> \nu_k(\Omega_{si_2\ldots
i_k})\geq 0, \quad si_2\ldots i_k \in N_{+,k}$$
 and therefore $\mu_k^\infty(\Omega_{si_2\ldots i_k})>0, \ \nu^\infty_k(\Omega_{si_2\ldots i_k})=0$,
 where
$$N_{+,k}:=\{i_1i_2\ldots i_k|p_{i_1i_2\ldots i_k}>r_{i_1i_2\ldots
i_k}\}.$$
\end{theorem}

\begin{proof}
  Since $\mu\neq\nu$ and $0<p_s<r_s$ there exists $m\neq s$\ such that
  $p_m>r_m\geq 0$.  If $r_m=0$, then $p_{sm}=p_s\cdot
  p_m>r_{sm}=r_s\cdot r_m=0$.  Therefore
  $\mu_2(\Omega_{sm})>\nu_2(\Omega_{sm})$ and
  $\mu_2^\infty(\Omega_{sm})>0, \ \nu_2^\infty(\Omega_{sm})=0$. In
  this case the theorem is proved.

If  $p_m>r_m\geq 0$ and $r_m\neq0$, then $p_m/r_m>1$ and $(p_m/r_m)^k\to \infty$,  with $ k\to \infty$.
Thus   $(p_m/r_m)^k> p_s/r_s$ and
 $p_s\cdot (p_m)^k >r_s\cdot (r_m)^k$  for some finite $k$. Therefore, if for all
  $i_2=m,\,\ldots ,\,i_k=m$ the conditions
 $p_{si_2\ldots i_k}>r_{si_1\ldots i_k}>0$ hold, then $si_2\ldots i_k\in N_{+,k}$,
 and  by Theorem \ref{3thCnlt} we get
$$p_{sm\ldots m}^\infty=\mu_k^\infty(\Omega_{si_2\ldots i_k})>0,
\quad  r_{sm\ldots m}^\infty=\nu_k^\infty(\Omega_{si_2\ldots
i_k})=0,  \quad  i_2=m,\quad \ldots ,\quad i_k=m.  $$
\end{proof}

We are able to estimate the maximal value of
$\mu_k^\infty(\Omega_{si_2\ldots i_k})>0$ in Theorem  \ref{4sss}.

\begin{prop}
   $$\max\limits _k \mu_k^\infty(\Omega_{si_2\ldots i_k})\leq p_s.$$
\end{prop}

\begin{proof}
By Theorem \ref{3thCnlt},
$$
\begin{aligned}
\mu_k^\infty(\Omega_{si_2\ldots i_k}) & =1/D_k \sum\limits_{s{i_2}
\ldots {i_k}\in{N_{+,k}}}(p_{s{i_2}\ldots {i_k}}-r_{s{i_2}\ldots
{i_k}})
\\
& =1/D_k \sum\limits_{s{i_2} \ldots
{i_k}\in{N_{+,k}}}(p_s(p_{i_2}\cdots p_{i_k}-r_{i_2}\cdots
r_{i_k}) -(r_s-p_s)\cdot r_{i_2}\cdots r_{i_k})
\\
 & \leq p_s/D_k
\sum\limits_{{i_2} \ldots {i_k}\in{N_{+,k-1}}}(p_{{i_2}\ldots
{i_k}}-r_{{i_2}\ldots {i_k}}) \leq p_s/D_k \cdot D_{k-1}\leq
p_s,
\end{aligned}
$$
where we used the following lemma.
\end{proof}

\begin{lemma}\label{4Dgr} For any couple of measures $\mu,\nu \in {\cal M}^{\rm ss}(\Omega),
\ \mu\neq\nu$ the variation distance $D_k=D(\mu_k,\nu_k)$ between their $k$th structural
approximated variants $\mu_k,\nu_k$ creates a non-decreasing
sequence
$$D_k\leq D_{k+1}, \quad k\geq 1.$$
\end{lemma}
\begin{proof}
By definition,
$D_k=1/2 \sum\limits_{{i_1}{i_2}\ldots {i_k}=1}^{n}|p_{{i_1}{i_2}\ldots {i_k}}-r_{{i_1}{i_2}\ldots {i_k}}|$.
Since $$1=p_{(k+1)i_1}+\cdots +p_{(k+1)i_n}=r_{(k+1)i_1}+\cdots +r_{(k+1)i_n}$$ we have
$$D_k=1/2\sum\limits_{{i_1}\ldots {i_k}=1}^{n}|p_{{i_1}\ldots {i_k}}(p_{(k+1)i_1}+\cdots +p_{(k+1)i_n})
-r_{{i_1}\ldots {i_k}}(r_{(k+1)i_1}+\cdots +r_{(k+1)i_n})|$$
$$\leq 1/2\sum \limits_{{i_1}\ldots {i_{k+1}}=1}^{n}|p_{{i_1}\ldots {i_{k+1}}}-r_{{i_1}\ldots {i_{k+1}}}|=D_{k+1}$$
\end{proof}

\begin{theorem}\label{6thg}
  Given a couple of the similar structure measures $\mu,\nu \in {\cal
    M}^{\rm ss}(\Omega)$ assume $\Omega_{i_1\ldots i_k}=(1/n)^k,
  \forall k$ and
 \begin{equation}\label{1cond}
{\rm supp} \mu=\Omega={\rm supp} \nu. \end{equation}
Then for any $0<\varepsilon <1 $ there exist controlled
structural redistributions of the starting measure $\mu$ such that on
$k$th step, \ $ 1\leq k < \infty$, of the rough controlled
approximations $\mu_k \to \tilde \mu_k, \ \nu_{k} \to \nu_{k}$ the
limit conflict state $\{ \tilde \mu_{k}^\infty, \nu_{k}^\infty \}$
obeys the properties
 \begin{equation}\label{apr}
 {\rm supp} {\tilde \mu}_{k}^\infty=\Omega_{k,+}, \quad \lambda(\Omega_{k,+})\geq 1-\varepsilon,
 \quad
{\rm supp} { \nu}_{k}^\infty=\Omega_{k,-}, \quad \lambda(\Omega_{k,-}) \leq \varepsilon,
\end{equation}
where
$$\Omega= \Omega_{k,+}\bigcup \Omega_{k,-}$$ denotes  the Hahn--Jordan decomposition corresponding
to the
signed measure $\tilde \omega_k=\tilde \mu_k-\nu_k$.
\end{theorem}

\begin{proof}
Without loss of generality we put $q_{ki}=1/n$ for all $k,i$
and assume that all $p_{ki_k}\neq 0 \neq r_{ki_k}$.
Moreover, we can assume that for some single  $1\leq s \leq n$
\begin{equation}\label{4dc}
p_{ks}< r_{ks} \quad {\rm and} \quad
p_{ki_k}> r_{ki_k}, \quad  i_k \neq s, \quad \forall  k\geq 1.\end{equation}
 Then at the first step of the
 rough approximation we have
 $${\rm supp} \mu^\infty_1=\Omega_{1,+}=\bigcup_{i_1\neq s} \Omega_{i_1}, \
 \ \lambda(\Omega_{i_1})=1-1/n.$$ So, the theorem is proved if $1/n <
 \varepsilon$. Assume $(1/n)^2 < \varepsilon < 1/n$.  Show that
 (\ref{apr}) may be reached at the second step of the rough controlled
 approximation.  Indeed, replace all $p_{1i_1}, \ i_i\neq s$ with
 $\tilde p_{1i_1} =r_{1i_1}+\delta/(n-1)$ where $\delta$ is chosen so
 that $\tilde p_{1s}= r_{1s}-\delta$ satisfies the conditions
$$(n-1) n^{-1} r_{1s} <  \tilde p_{1s} < r_{1s}$$
and
$$\tilde p_{1s} \tilde p_{2i_2}>r_{1s}r_{2i_2}, \quad  i_2\neq s,
\quad
\tilde p_{1s} \tilde p_{2s}< r_{1s}r_{2s}.$$
Then for the controlled conflict  with the division
$$\Omega=\bigcup_{i_1\neq s}\Omega_{i_1}\bigcup_{i_2=1}^n\Omega_{si_2}$$
for $\tilde \mu_2, \nu_2$ we obtain
$${\rm supp} (\tilde \mu^\infty_2)=\Omega_{2,+}=\bigcup_{i_1\neq s}
\Omega_{i_1}\bigcup_{i_2\neq s}\Omega_{si_2},
 \quad \lambda(\Omega_{2,+})=1-(1/n)^2> 1-\varepsilon,$$
 $${\rm supp} (\nu^\infty_2)=\Omega_{2,-}=
 \Omega_{ss}, \quad \lambda(\Omega_{ss})=(1/n)^2<\varepsilon.$$ If
 $(1/n)^3 < \varepsilon < (1/n)^2$, then one can reach (\ref{apr}) at
 $k=3$ step of the rough controlled approximation with an appropriate
 $\delta$. And so on.
\end{proof}

\section{\bf Discussion}

Let us discuss some interpretation of the above results from the point
of view of possible applications.

We recall that for models of DSC which describe natural conflicts its
trajectories are governed by some law of conflict interaction which is
independent of the time.  In this case each limit state
$\{\mu^\infty,\nu^\infty \}$ of the system is a fixed point (an
equilibrium state) defined by the starting couple of measures $\mu,
\nu \in {\cal M}^{+}_1(\Omega)$ \ (see Theorems \ref{3thCnlt}).

In other situation when the law of conflict interaction may be changed
at any moment of time, we deal with the {\it controlled conflict}. In
other words, such changes mean the choice of the local strategies.  In
the present paper we discussed only the simplest versions of models
with the controlled conflict.  They were reduced to redistributions of
the starting measures $\mu, \nu \in {\cal M}^{\rm ss}(\Omega)$ at
$k$th steps of their structural approximations and mean changes of the
vectors $p_l, \ l \leq k$ in the matrix $P$ (see Lemma \ref{2asm}) to
other vectors ${\tilde p}_l \in {\mathbb R}^n_{+,1}$. These changes
where aimed to get the new limit states $\{\tilde \mu^\infty,\tilde
\nu^\infty \}$ different from the ones in the case of natural
conflict.

In Section 4 it was shown that the limit result of the natural
conflict may be essentially changed. So, in the situation of complete
defeat for the opponent $A$ in a region $\Omega_s$, when from the
condition $0<\mu(\Omega_s)<\nu(\Omega_s)$ it follows that
$\mu^\infty(\Omega_s)=0$, i.e., there appears a limit gap, one can
produce some redistribution of $\mu$ inside $\Omega_s$ in such a way
that after the controlled conflict interaction, $A$ reaches a victory
in a subregion $\tilde \Omega\subset \Omega_s$, i.e., ${\tilde
  \mu}^\infty({\tilde \Omega})>0$. Moreover, in Theorem \ref{4thbig}
we got estimates both for the value of ${\tilde \mu}^\infty({\tilde
  \Omega})$ and for the size of area with a limit priority, i.e., for
$\lambda({\tilde \Omega})$.  Clearly, these estimates depend on the
structure of divisions (\ref{1dq}) and the values of
redistributions. Theorem \ref{6thg} shows that under an appropriate
redistribution for one of the starting measure, the limit result of
the controlled conflict might be very successful for one opponent and
extremely bad for the other (it looses almost whole territory, see
(\ref{apr})).

\end{document}